\documentclass[12pt,a4paper,psamsfonts]{amsart}
\usepackage{amssymb,amscd,amsxtra,calc}
\usepackage{cmmib57}

\theoremstyle{plain}
    \newtheorem{thm}{Theorem}[section]
    
    \newtheorem{claim}[thm]{Claim}
    
    \newtheorem{lemma}[thm]{Lemma}
    
    \newtheorem{question}[thm]{Question}
    \newtheorem{theorem}[thm]{Theorem}

\theoremstyle{definition}
    \newtheorem{definition}[thm]{Definition}

    \newtheorem{remark}[thm]{Remark}
\theoremstyle{remark}

    \newtheorem{setup}[thm]{}

\newcommand{\BCC}{\mathbb{C}}

\newcommand{\BPP}{\mathbb{P}}

\newcommand{\BRR}{\mathbb{R}}
\newcommand{\BZZ}{\mathbb{Z}}

\newcommand{\SC}{\mathcal{C}}

\newcommand{\SK}{\mathcal{K}}
\newcommand{\SO}{\mathcal{O}}

\newcommand{\alb}{\operatorname{alb}}
\newcommand{\Aut}{\operatorname{Aut}}

\newcommand{\GL}{\operatorname{GL}}

\newcommand{\id}{\operatorname{id}}
\newcommand{\Imm}{\operatorname{Im}}
\newcommand{\Ker}{\operatorname{Ker}}

\newcommand{\NS}{\operatorname{NS}}

\newcommand{\torsion}{\operatorname{torsion}}

\newcommand{\Alb}{\operatorname{Alb}}
\newcommand{\OH}{\operatorname{H}}

\newcommand{\ratmap}
{{\,\cdot\negmedspace\cdot\negmedspace\cdot\negmedspace\to\,}}

\begin{document}

\title[A theorem of Tits type for compact K\"ahler manifolds]{
A theorem of Tits type for compact K\"ahler manifolds}

\author{De-Qi Zhang}
\address
{
\textsc{Department of Mathematics} \endgraf
\textsc{National University of Singapore, 2 Science Drive 2,
Singapore 117543}}
\email{matzdq@nus.edu.sg}

\begin{abstract}
We prove a theorem of Tits type for compact K\"ahler manifolds,
which has been conjectured in the paper \cite{KOZ}.
\end{abstract}


\thanks{The author is supported by an ARF of NUS}

\maketitle

\section{Introduction}

We work over the field $\BCC$ of complex numbers.
In this note, we prove first the following result, which also gives an affirmative
answer to the {\it conjecture of Tits type for compact K\"ahler manifolds}
as formulated in \cite{KOZ}.

\begin{theorem}\label{ThA}
Let $X$ be an $n$-dimensional $(n \ge 2)$
compact K\"ahler manifold and $G$ a subgroup of $\Aut(X)$.
Then one of the following two assertions holds:
\begin{itemize}
\item[(1)]
$G$ contains a subgroup isomorphic to the non-abelian free group $\BZZ * \BZZ$,
and hence $G$ contains subgroups isomorphic to non-abelian free groups of all countable ranks.
\item[(2)]
There is a finite-index subgroup $G_1$ of $G$ such that
the induced action $G_1 | H^{1,1}(X)$ is solvable and Z-connected.
Further, the subset
$$N(G_1) := \{g \in G_1 \,\, | \,\, g \,\,\, \text{\rm is of null entropy}\}$$
of $G_1$ is a normal subgroup of $G_1$ and the quotient group $G_1/N(G_1)$ is a
free abelian group of rank $r \le n-1$ {\rm (see \ref{ThB} below for the boundary cases)}.
\end{itemize}
\end{theorem}

In Theorems \ref{ThA} and \ref{ThB}, the action $G | H^{1,1}(X)$ is
{\it Z-connected} if its Zariski-closure in $\GL(H^{1,1}(X))$ is connected
with respect to the Zariski topology.
$X$ is said to be {\it almost homogeneous} (resp. {\it dominated by} some closed subgroup $H$
of the identity connected component $\Aut_0(X)$ of $\Aut(X)$)
if some $\Aut_0(X)$-orbit (resp. $H$-orbit) is dense open in $X$.
A compact K\"ahler manifold is {\it weak Calabi-Yau}
if the irregularity $q(X) := h^1(X, \SO_X) = 0$ and the {\it Kodaira dimension}
$\kappa(X) = 0$. A projective manifold is {\it ruled} if it is
birational to $\BPP^1 \times$ (another projective manifold).

\begin{theorem}\label{ThB}
Let $X$ be an $n$-dimensional $(n \ge 2)$
compact K\"ahler manifold and $G$ a subgroup of $\Aut(X)$
such that the induced action $G|H^{1,1}(X)$ is solvable and
Z-connected. Then the set $N(G)$ as in $\ref{ThA}$ is a normal subgroup of
$G$ such that $G/N(G) \cong \BZZ^{\oplus r}$ with $r \le n-1$.
If $r = n-1$, then the algebraic dimension
$a(X) \in \{0, n\}$ $(${\rm cf.}~\cite[3.2]{Ue}$)$,
the anti-Kodaira dimension $\kappa(X, -K_X) \le 0$
and one of Cases $(1) \sim (4)$ below occurs.
\begin{itemize}
\item[(1)]
$X$ is bimeromorphic to a complex torus.
\item[(2)]
$X$ is a weak Calabi-Yau K\"ahler manifold.
\item[(3)]
$\Aut_0(X) = (1)$, $q(X) = 0$ and the Kodaira dimension
$\kappa(X) = -\infty$. Further,
$X$ is rationally connected provided that $X$ is projective and uniruled.
\item[(4)]
$\Aut_0(X)$ is a non-trivial linear algebraic group. $X$ is an almost homogeneous
projective manifold and dominated
by every positive-dimensional characteristic closed subgroup of $\Aut_0(X)$.
Hence $X$ is unirational and ruled.
In particular, $X$ is a rational variety,
unless $\dim X \ge 4$ and $\Aut_0(X)$ is semi-simple.
\end{itemize}
\end{theorem}

A remark on the uniruledness in Theorem \ref{ThB} (3):
for a projective manifold $X$, the {\it good minimal model conjecture}
claims that $X$ is uniruled if and only if
$\kappa(X) = -\infty$; this conjecture has been confirmed
when $\dim X \le 3$; see \cite[\S 3.13]{KM}.

We refer to \cite{Gr}, \cite{Yo}, \cite{Fr}, \cite{DS04} and \cite{DS05}
for the definitions of the $i$-th {\it dynamical degrees} $d_i(g)$ and
the {\it topological entropy} $h(g)$ for $g \in \Aut(X)$. It is known that
$$h(g) = \max_{1 \le i \le n} \log d_i(g)
= \max \{\log |\lambda| \, ; \, \lambda \,\, \text{is an eigenvalue of} \,\,
g^* \, | \, (\oplus_{i \ge 0}H^i(X, \BCC))\};$$
see also \cite[Proposition 5.7]{Di}.
And $g \in \Aut(X)$ is of {\it positive entropy} (resp. {\it null entropy})
if and only if $d_1(g) > 1$ or equivalently $h(g) > 0$ (resp. $d_1(g) = 1$ or equivalently
$h(g) = 0$).

Theorem \ref{ThA} is an analogue to the famous {\it Tits alternative theorem} \cite[Theorem 1]{Ti}:
For a subgroup $H \le \GL_m(\BCC)$ with $m \ge 1$, either
\begin{itemize}
\item[(i)]
$H$  contains a subgroup isomorphic to $\BZZ * \BZZ$, or
\item[(ii)]
$H$ contains a solvable subgroup of finite index.
\end{itemize}

When $G$ is abelian, Theorem \ref{ThA} follows from \cite{DS04}.
The key ingredients of our proof are: the very inspiring results of Dinh-Sibony \cite{DS04},
especially the Hodge-Riemann type result \cite[Corollary 3.5]{DS04},
the theorem of Lie-Kolchin type for a cone
\cite[Theorem 1.1]{KOZ}, and the trick of considering the
'quasi nef sequence' in \ref{Ms}.

\begin{remark}
$ $
\begin{itemize}
\item[(1)]
For a projective manifold $X$, if we replace $H^{1,1}(X)$ by
$\NS(X) \otimes_{\BZZ} \BCC$,
the same statements of Theorems \ref{ThA}
and \ref{ThB} hold; to prove them, we just replace $H^2(X, \BZZ)$ and the K\"ahler
cone by respectively $\NS(X)/(\torsion)$ and the ample cone.
\item[(2)]
For Theorem \ref{ThB}, it is easy to show that $r < h^{1,1}(X)$.
However, $h^{1,1}(X)$ has no upper bound even for surfaces.
Our bound $r \le n-1$ is optimal as seen in \cite[Example 4.5]{DS04},
where their examples $X$ are abelian varieties (and hence $\kappa(X) = 0$).
Theorem \ref{ThB} shows that, indeed, such optimal cases happen only when
$\kappa(X) \le 0$. More generally, one has $r \le \max\{0, \, n - 1 - \kappa(X)\}$ (optimal!)
by Lemma \ref{kappa}.
\item[(3)]
When $X$ is a surface, Theorem \ref{ThA} says that either $G$ is very
chaotic, or has at most one dynamically interesting symmetry.
\item[(4)]
The proofs of Theorems \ref{ThA} and \ref{ThB} are straightforward, with no
any fancy stuff.
\end{itemize}
\end{remark}

\noindent
{\bf What are the obstacles} (in proving Theorem \ref{ThB} for non-commutative $G$)?

First, without the commutativity, one could not produce enough common {\it nef} (as usually
desired) eigenvectors $L_i$ of $G$ in the closure $\overline{\SK}(X)$
of the K\"ahler cone of $X$,
in order to construct a homomorphism $\psi: G \to \BRR^{n-1}$.
Second, even if one has $n-1$ of such $L_i$, the non-vanishing of
$L_1 \cdots L_{n-1}$ (needed to show the injectivity of $\psi$ modulo $N(G)$)
could not be checked, unless, being lucky enough, $L_i$'s give rise to pairwise
distinct characters of $G$.
A naive approach is to use induction on $h^{1,1}(X)$
and apply the generalized Perron-Frobenius theorem in \cite{Bi} to the cone:
$\Imm(\overline{\SK}(X) \to H^{1,1}(X, \BRR)/(\BRR L_1))$;
see \ref{ne} for the notation. However, the latter cone may not be
closed and hence a common eigenvector of $G$ in the closure of the latter cone
is of the form $L_2 + \BRR L_1$ (a coset) but with $L_2$ {\it not necessarily} in the cone $\overline{\SK}(X)$.
Fortunately, these difficulties are taken care by considering the
quasi nef sequence in \ref{Ms}, and the latter is further exploited in Lemma \ref{reduction}.

To the best of our knowledge, there
are only partial solutions to the conjecture of Tits type for complex varieties,
like \cite{KOZ}
and \cite{Ki}.
See also \cite{Og} and \cite{Z1} for related results.

\par \vskip 1pc \noindent
{\bf Acknowledgement}

I thank A.~Fujiki for the discussion on
algebraic reduction, K.~Oguiso for encouraging me to generalize the result
to the K\"ahler case, T.-C.~Dinh and N.~Sibony
for the comments and for pointing out \cite[page 301]{DS04}
related to \ref{Ms} below, and the referee for the suggestions.

\section{Proof of Theorems \ref{ThA} and \ref{ThB}}

\begin{setup}\label{ne}
{\bf Numerically equivalent cohomology classes and K\"ahler $(k, k)$-forms}
\par
Let $X$ be a compact K\"ahler manifold of dimension $n$.
Denote $H^{i,i}(X, \BRR) = H^{i,i}(X) \cap H^{2i}(X, \BRR)$.
We use letters of upper case like $D, L, M$ to denote elements of $H^{i,i}(X)$.
By abuse of notation, the {\it cup product} $L \cup M$ for $L \in H^{i,i}(X)$
and $M \in H^{j,j}(X)$ is denoted as $L . M$ or simply $L \, M$.

Let $N^r(X)$ be the quotient of $H^{r,r}(X, \BRR)$ modulo
numerical equivalence: for $D_1, D_2 \in H^{r,r}(X, \BRR)$,
$D_1 \equiv D_2$ ({\it numerically equivalent})
if and only if $(D_1 - D_2) \, L_1 \cdots L_{n-r} = 0$
for all $L_i$ in $H^{1,1}(X, \BRR)$.
Denote by $[D_1] \in N^r(X)$ the class containing $D_1$,
but {\it we will write} $D_1 \in N^r(X)$ by abuse of notation.
Notice that numerical equivalence may {\it not} coincide with
cohomological equivalence when $r \le n-2$.

The {\it cone} $\overline{\SK}(X) \subset H^{1,1}(X, \BRR)$ is defined to be
the closure of the {\it K\"ahler cone} $\SK(X)$ of $X$.
Elements of $\overline{\SK}(X)$ are said to be {\it nef}.

We now use the definitions
of \cite[\S A1]{NZ}. Let \( \omega \) be a \( C^{\infty} \)-\( (k, k) \)-form on \( X \)
with \( 1 \leq k \leq n \).
For a local coordinate system \( (z_{1}, z_{2}, \ldots, z_{n}) \) of \( X \),
our \( \omega \) is locally expressed as
\[ \omega = (\sqrt{-1})^{k^{2}}
\sum\nolimits_{I, J \subset \{1, 2, \ldots, n\}}
a_{I, J} d z_{I} \wedge d \bar{z}_{J} \]
with \( C^{\infty} \)-functions \( a_{I, J} \), where
\( \sharp I = \sharp J = k \) and
\( d z_{I} := d z_{i_{1}} \wedge d z_{i_{2}} \wedge \cdots \wedge d
z_{i_{k}} \)
with \( I = \{ i_{1}, i_{2}, \ldots, i_{k}\} \) and \( i_{1} <
i_{2} < \cdots < i_{k} \).
This \( \omega \) is called a \emph{K\"ahler} \( (k, k) \)-form if
it is \( d \)-closed (\( d \omega = 0 \)) and
real (\( \overline{\omega} = \omega \)), and if the matrix
\( (a_{I, J}) \) is positive-definite everywhere on \( X \).
Note that this is just the usual definition of K\"ahler
form when \( k = 1 \).
Inside of the real vector space
\( \OH^{k, k}(X, \BRR) \),
the set \( P^{k}(X) \) of the classes \( [\omega] \) of
K\"ahler \( (k, k) \)-forms \( \omega \) on \( X \)
is a strictly convex open cone. It is called the
\emph{K\"ahler cone of degree} \( k \).
Its closure in \( \OH^{k, k}(X, \BRR) \) is
denoted as \( \overline{P^{k}(X)} \).
\end{setup}

\begin{setup}\label{Ms} {\bf The quasi nef sequence}.
Let $X$ be a compact K\"ahler manifold of dimension $n$.
Consider nonzero elements
$M_1, \dots, M_s$ (with $1 \le s \le n$) in $H^{1,1}(X, \BRR)$
satisfying:
\end{setup}

\begin{itemize}
\item[(i)]
$M_1$ is nef, and
\item[(ii)]
For every $2 \le r \le s$,
there are nef $L_r(j) \in \overline{\SK}(X)$ such that
$$H^{r,r}(X, \BRR) \setminus \{0\} \, \ni \, M_1 \cdots M_r \ = \
\lim_{j \to \infty} (M_1 \cdots M_{r-1}  \, L_r(j)).$$
So $M_1 \cdots M_r \ne 0$ in $N^r(X)$ by Lemma \ref{pre} (1) below.
\end{itemize}

\begin{lemma}\label{pre}
Let $X$ and $M_r$ $(1 \le r \le s)$ be as in $\ref{Ms}$. Then the following are true.
\begin{itemize}
\item[(1)]
For all $i_1 \le \cdots \le i_t \le s$
with $s+t \le n$, we have {\rm (cf. \ref{ne})}:
$$M_1 \cdots M_s \in \overline{P^{s}(X)}, \hskip 1pc (M_1 \cdots M_{s}) \,
(M_{i_1} \cdots M_{i_t}) \in \overline{P^{s+t}(X)}.$$
In particular, $M_1 \cdots M_s \, L_{s+1} \cdots L_n$ $> 0$ for all
$L_i$ in the K\"ahler cone $\SK(X)$, and hence $M_1 \cdots M_s \ \ne \ 0$ in $N^s(X)$.

\item[(2)]
Suppose $1 \le s \le n-1$ and all
$L_{s}, \dots, L_{n-1} \in \SK(X)$.
Then the quadratic form
$$q(x, y) := -M_1 \dots M_{s-1} \, L_{s} \cdots L_{n-2} \, x \, y$$
is semi-positive over the set
$$P(M_1 \cdots M_{s-1} \, L_{s} \cdots L_{n-1}) :=
\{z \in H^{1,1}(X, \BRR) \,\, | \,\, M_1 \cdots M_{s-1} \, L_{s} \cdots L_{n-1} \, z = 0\}$$
of $M_1 \cdots M_{s-1} \, L_{s} \cdots L_{n-1}$-primitive classes; see \cite[\S 3]{DS04}
for relevant material.
\item[(3)]
We have $q(M_{s}, M_{s}) \le 0$.

\par
In the following, assume $1 \le s \le n-1$, and take $M_s' \in H^{1,1}(X, \BRR)$ such that
$$H^{s,s}(X, \BRR) \setminus \{0\} \, \ni \, M_1 \cdots M_{s-1} \, M_s' \ = \
\lim_{j \to \infty} (M_1 \cdots M_{s-1} \, L_s(j)')$$
for some $L_s(j)' \in \overline{\SK}(X)$.
So this $M_s'$ enjoys the properties of $M_s$.

\item[(4)]
Suppose that
$M_1 \cdots M_{s-1} \, M_s \, M_s' = 0$ in $N^{s+1}(X)$.
Then there is exactly one nonzero real number $b$ such that
$M_1 \cdots M_{s-1} \, (M_s + bM_s') = 0$ in $N^{s}(X)$.
\item[(5)]
Suppose that $h : X \to X$ is a surjective endomorphism
such that
$$h^*(M_1 \cdots M_{s-1} \, M) \ = \ \lambda(M) \, (M_1 \cdots M_{s-1} \, M)$$
for both $M = M_s$ and $M_s'$
with real numbers $\lambda(M_s) \ne \lambda(M_s')$. Then
we have $M_1 \cdots M_{s-1} \, M_s \, M_s' \ \ne \ 0$ in $N^{s+1}(X)$.

\end{itemize}
\end{lemma}

\begin{proof}
(1)We only need to prove the first part and for this we
proceed by induction on $t$. The case $t = 0$ is clear.
Suppose that we are done with the case $t$. Then for the case $t+1$ (with $i_{t+1} = k \le s$ say)
the assertion (1) follows from:
$$\begin{aligned}
(M_1 \cdots M_s) \, (M_{i_1} \cdots M_{i_{t+1}}) \ = \
&\lim_{j \to \infty} ((M_1 \cdots M_{k-1} \, L_k(j) \, M_{k+1} \cdots M_s) \,
(M_{i_1} \cdots M_{i_{t+1}})) \\
\ = \ &\lim_{j \to \infty} ((M_1 \cdots M_s) \,(M_{i_1} \cdots M_{i_{t}} \, L_k(j))).
\end{aligned}$$

The assertion (2) follows from the Hodge-Riemann Theorem of Gromov; see its
other variation \cite[Corollary 3.4]{DS04}. (3) is from (1).
The assertions (4) and (5) are actually proved in \cite[Corollary 3.5, Lemma 4.4]{DS04},
but we use (2) and (3) (see also \cite[Remark 3.6]{DS04})
and note that $h^*$ induces an automorphism of $H^{i,i}(X, \BRR)$ ($i \ge 0$),
acts as a scalar multiple by $\deg(h)$ on $H^{2n}(X, \BRR)$
(and also preserves the K\"ahler cone of $X$).
\end{proof}

\begin{definition}
Let $V$ be a free $\BZZ$-module of finite rank or a real vector space of finite dimension.
Let $V_{\BCC}$ be its complexification.
A subgroup $H \le \GL(V)$ is {\it Z-connected on} $V$ or {\it on} $V_{\BCC}$, if the Zariski-closure
in $\GL(V_{\BCC})$ of (the natural extension) $H | V_{\BCC}$
is connected with respect to the Zariski topology.

Let $X$ be a compact K\"ahler manifold. For a subgroup $G \le \Aut(X)$,
we consider its action (by pullback) $G|\Lambda$
on some $G$-stable subspace $\Lambda \subset H^{i,i}(X, \BRR)$.
We say that $G$ is {\it Z-connected} on $\Lambda$ if
$G | \Lambda$ is Z-connected on $\Lambda$;
our $G$ is {\it solvable on} $\Lambda$ if $G | \Lambda$
is solvable.
\end{definition}

\begin{remark}\label{Zcr1}
By the Tits alternative theorem \cite[Theorem 1]{Ti},
either $G | \Lambda_{\BCC}$ and hence $G$ contain a subgroup isomorphic
to $\BZZ * \BZZ$, or
$G|\Lambda_{\BCC}$ contains a solvable subgroup $H$ of finite index.
In the latter case, the preimage $G_1 \le G$ of
the identity component of the closure
of $H$ in $\GL(\Lambda_{\BCC})$,
is of finite-index in $G$, and is solvable and Z-connected on $\Lambda$.
\end{remark}

We shall use the theorem of Lie-Kolchin type for a cone in \cite{KOZ}:

\begin{theorem}\label{LK} {\rm (cf. \cite[Theorem 1.1]{KOZ})}
Let $V$ be a finite-dimensional real vector space and $\{0\} \ne C \subset V$ a strictly convex
closed cone. Suppose that a solvable subgroup $G \le \GL(V)$ is
Z-connected on $V_{\BCC}$ and $G(C) \subseteq C$.
Then $G$ has a common eigenvector in the cone $C$.
\end{theorem}

\begin{setup}
{\bf Proof of Theorems \ref{ThA} and \ref{ThB}}
\end{setup}

By Remark \ref{Zcr1}, we only need to show Theorem \ref{ThB}.
From now on till Lemma \ref{algdim}, we assume that $G|H^{1,1}(X)$
{\it is solvable and Z-connected}.

Applying Theorem \ref{LK} to the cone $\overline{\SK}(X) \subset H^{1,1}(X, \BRR)$,
there is a common eigenvector $M_1 \in \overline{\SK}(X)$ of $G$.
We may write $g^*M_1 = \chi_1(g) M_1$ where $\chi_1: G \to \BRR_{> 0}$ is a character.
Consider the induced action of $G$ on the finite-dimensional subspace
$M_1 . H^{1,1}(X, \BRR)$ ($\subseteq H^{2,2}(X, \BRR)$),
which is nonzero by Lemma \ref{pre} (1).
Our $M_1 . \overline{\SK}(X)$ is a cone and spans $M_1 . H^{1,1}(X, \BRR)$
but $(*)$: it may not be {\it closed} in $M_1 . H^{1,1}(X, \BRR)$.
Now $G$ is solvable and Z-connected on $M_1 . H^{1,1}(X, \BRR)$,
because $G | (M_1 . H^{1,1}(X, \BCC))$ is the image of $G | H^{1,1}(X, \BCC)$
(multiplied by $\chi_1$).
Applying Theorem \ref{LK} to the closure of $M_1 . \overline{\SK}(X)$ in $M_1 . H^{1,1}(X, \BRR)$,
we get $M_2 \in H^{1,1}(X, \BRR)$ such that $M_1 \, M_2$ is a common eigenvector of $G$ and
\ref{Ms} holds for $s = 2$. Notice that $M_2$ {\it may not be nef} because of the $(*)$ above.
We may write
$g^*(M_1 \, M_2) = \chi_1(g) \chi_2(g) (M_1 \, M_2)$
where $\chi_2 : G \to \BRR_{> 0}$ is a character.
Inductively, we obtain $M_1, \dots, M_n$
such that for every $1 \le r \le n-1$ our
$G$ is solvable and Z-connected on $M_1 \cdots M_r . H^{1,1}(X, \BRR)$,
that \ref{Ms} holds for $s = n$
and that for all $1 \le t \le n$
$$g^*(M_1 \cdots M_t) = \chi_1(g) \cdots \chi_t(g) \, (M_1 \cdots M_t)$$
where $\chi_i : G \to \BRR_{> 0}$ ($1 \le i \le n$) are characters.
Of course, $\chi_1(g) \cdots \chi_n(g) = \deg(g) = 1$.

Define a homomorphism into the additive group as follows
$$\psi : G \to \BRR^{n-1}, \hskip 2pc g \mapsto (\log \chi_1(g), \, \dots, \, \log \chi_{n-1}(g)).$$
\par \noindent
Now the first part of Theorem \ref{ThB} follows from the two claims below.

\begin{claim}\label{pfc1}
$\Ker(\psi) = N(G)$.
\end{claim}

\begin{proof}
By the equivalent definition of the entropy $h(g)$ in the introduction,
we have $N(G) \subseteq \Ker(\psi)$.
Suppose the contrary that $g \in \Ker(\psi)$ is of positive entropy.
Then $\chi_i(g) = 1$ for all $1 \le i \le n$, but
the first dynamical degree $d_1(g) > 1$. It is known that
$d_1(g) = \max \{|\lambda| \, ; \, \lambda$ is an eigenvalue of $g^* | H^{1,1}(X, \BRR) \}$.
Applying the generalized Perron-Frobenious theorem in \cite{Bi} to the cone $\overline{\SK}(X)$
which spans $H^{1,1}(X, \BRR)$, we find a nonzero $L_1$ in
$\overline{\SK}(X)$ such that $g^*L_1 = \lambda_1 L_1$
with $\lambda_1 = d_1(g) > 1$.
Now
$$M_1 \cdots M_{n-1} \, L_1 = g^*(M_1 \cdots M_{n-1} \, L_1) =
\lambda_1 \, (M_1 \cdots M_{n-1} \, L_1).$$
Since $\lambda_1 \ne 1$, we have $M_1 \cdots M_{n-1} \, L_1 = 0$.
If $M_1 \, L_1 = 0$ in $N^2(X)$ then $M_1$ and $L_1$ are parallel in $N^1(X)$
by \cite[Corollary 3.2]{DS04}
or Lemma \ref{pre} (5). Hence $1 = \chi_1(g) = \lambda_1 > 1$, a contradiction.
Take $t \ge 2$ minimal such that $M_1 \cdots M_t \, L_1 = 0$ in $N^{t+1}(X)$.
Then both $M_1 \cdots M_{t-1} \, L_1$ and $M_1 \cdots M_{t-1} \, M_t$
are nonzero in $N^t(X)$; see Lemma \ref{pre} (1).
Since $L_1$ is nef, we can apply Lemma \ref{pre} (5)
to $M_t' := L_1$ and conclude that $M_1 \cdots M_t \, L_1 \ne 0$ in $N^{t+1}(X)$,
since $g^*(M_1 \cdots M_t) = M_1 \cdots M_t$ while $g^*(M_1 \cdots M_{t-1} \, L_1) =
\lambda_1 \, (M_1 \cdots M_{t-1} \, L_1)$ with $\lambda_1 \ne 1$.
We have reached a contradiction.
This proves Claim \ref{pfc1}.
\end{proof}

\begin{claim}\label{pfc2}
$\Imm(\psi)$ is discrete in $\BRR^{n-1}$ with respect to the Euclidean topology.
\end{claim}

\begin{proof}
Since $\psi$ is a homomorphism, it suffices to show that $(0, \dots, 0)$
is an isolated point of $\psi(G)$.
Let $\varepsilon > 0$. Consider the set $\Sigma = \Sigma_{\varepsilon}$
of all elements $g \in G$
satisfying $|\log \chi_i(g)| \le \varepsilon$ for all $1 \le i \le n-1$,
or equivalently $\chi_i(g^{\pm}) \le \delta := e^{\varepsilon}$.
Note that $g \in \Sigma$ if and only if $g^{-1} \in \Sigma$.
Let $g \in \Sigma$ with $\psi(g) \ne (0, \dots, 0)$. Then $g$ is of positive entropy
by Claim \ref{pfc1}. Thus there are nonzero $L_i$ in the cone $\overline{\SK}(X)$
such that $g^*L_i = \lambda_i L_i$ with $\lambda_1 = d_1(g) > 1$
and $\lambda_2^{-1} = d_1(g^{-1}) > 1$.

We now prove the assertion $(**)$: $d_1(g^{\pm}) \le \delta^{n-1}$.
It suffices to show $d_1(g^e) \le \delta$ for $e = 1$ or $-1$, because
$d_1(g^{\pm}) \le d_1(g^{\mp})^{n-1}$.
Indeed, applying the argument above to $\langle g \rangle$,
we get $M_1', \cdots, M_n'$ satisfying \ref{Ms} so that
each $M_1' \cdots M_t'$ is semi $g$-invariant, and we may take $M_1' = L_1, M_2' = L_2$
(for $L_1 \, L_2 \ne 0$ in $N^2(X)$ by Lemma \ref{pre} (5));
now $M_1' \cdots M_n' = g^*(M_1' \cdots M_n') = \lambda_1 \lambda_2 \cdots (M_1' \cdots M_n')$
implies that $1 = \lambda_1 \lambda_2 \cdots \le \lambda_1^{n-1} \lambda_2$
and hence $d_1(g^{-1}) = \lambda_2^{-1} \le d_1(g)^{n-1}$.
For the proof of $d_1(g^e) \le \delta$, notice
$$M_1 \cdots M_{n-1} \, L_i = g^*(M_1 \cdots M_{n-1} \, L_i) =
\chi_1(g) \cdots \chi_{n-1}(g) \, \lambda_i \,
(M_1 \cdots M_{n-1} \, L_i).$$
If $M_1 \cdots M_{n-1} \, L_i \ne 0$ for both $i = 1$ and $2$,
then $\chi_1(g) \cdots \chi_{n-1}(g) \, \lambda_i = 1$ and hence $1 < \lambda_1 = \lambda_2 < 1$,
absurd. Switching $g$ with $g^{-1}$ if necessary, we may assume that
$M_1 \cdots M_{n-1} \, L_1 = 0$. If $M_1 \, L_1 = 0$ in $N^2(X)$ then $M_1$ is parallel to
$L_1$ in $N^1(X)$ and hence $d_1(g) = \lambda_1 = \chi_1(g) \le \delta$.
If $M_1 \, L_1 \ne 0$ in $N^2(X)$ and $t \ge 2$ is minimal such that
$M_1 \cdots M_t \, L_1 = 0$ in $N^{t+1}(X)$, then
$d_1(g) = \chi_t(g) \le \delta$ by Lemma \ref{pre} (5);
see the proof of Claim \ref{pfc1}. The assertion $(**)$ is proved.

Therefore, for every $g \in \Sigma$,
the absolute values of all (complex) eigenvalues
of $g^{*}|H^2(X, \BZZ)$ are bounded by $\delta^{n-1}$,
whence Claim \ref{pfc2} holds;
see the proof of \cite[Corollary 2.2]{DS04}.
This also proves the first part of Theorem \ref{ThB}.
\end{proof}

We still need to get a better rank upper bound $r \le n-2$ except for the
four cases in Theorem \ref{ThB}. To distinguish, we also write $r = r(G)$.

\begin{lemma}\label{reduction}
Let $(X, G)$ be as in Theorem $\ref{ThB}$
and $Y$ a compact K\"ahler manifold with a biregular action by $G$ and
with $n = \dim X > k := \dim Y > 0$. Suppose that
$\pi: X \to Y$ is a $G$-equivariant surjective holomorphic map.
Then the rank $r(G) \le n - 2$.
\end{lemma}

\begin{proof}
We use the argument for the first part of Theorem \ref{ThB}, first
for $(Y, G)$ and then for $(X, G)$.
By the assumption,
both $G | H^{1,1}(X)$ and $G | H^{1,1}(Y)$ are solvable and Z-connected.
By the argument for the first part of Theorem \ref{ThB},
there are $M_1', \dots, M_k' \in H^{1,1}(Y, \BRR) $ satisfying \ref{Ms}
and the following for all $1 \le t \le k$
$$g^*(M_1' \cdots M_t') = \chi_1(g) \cdots \chi_t(g) \, (M_1' \cdots M_t').$$
We can continue the sequence $M_1 := \pi^*M_1', \cdots, M_k := \pi^*M_k'$
with $M_{k+1}, \cdots, M_n$ in $H^{1,1}(X, \BRR)$ so that $M_1, \dots, M_n$
satisfy \ref{Ms} and the following for all $1 \le t \le n$
$$g^*(M_1 \cdots M_t) = \chi_1(g) \cdots \chi_t(g) \, (M_1 \cdots M_t).$$

\par \noindent
Define a homomorphism into the additive group:
$$\varphi : G \to \BRR^{n-2}, \hskip 2pc g \mapsto (\log \chi_1(g), \, \dots, \, \log \chi_{k-1}(g),
\, \log \chi_{k+1}(g), \, \dots, \, \log \chi_{n-1}(g)).$$

As in the first part of
Theorem \ref{ThA}, we only need to show that $\Ker(\varphi) \subseteq N(G)$ and $\Imm(\varphi)$ is discrete.
Notice $(***)$: $\chi_1(g) \cdots \chi_k(g) = 1$ for all $g \in G$.

Suppose that $g \in \Ker(\varphi)$. Then $\chi_i(g) = 1$ holds also for $i = k$ by $(***)$, and hence for all
$1 \le i \le n-1$. Thus $\Ker(\varphi) = N(G)$ by the proof of Claim \ref{pfc1}.

Consider the set $\Sigma = \Sigma_{\varepsilon}$
of all elements $g \in G$
satisfying $|\log \chi_i(g)| \le \varepsilon$ for all $i$ with $i \ne k$ and $1 \le i \le n-1$,
or equivalently $\chi_i(g^{\pm}) \le \delta := e^{\varepsilon}$.
By $(***)$, $\chi_k(g^{\pm}) = (\chi_1(g) \cdots \chi_{k-1}(g))^{\mp} \le \delta^{k-1}$
and hence $\chi_i(g^{\pm}) \le \max \{\delta, \, \delta^{k-1}\}$ for all $g \in \Sigma$
and $1 \le i \le n-1$.
Now the proof of Claim \ref{pfc2} shows that $\varphi$ is discrete.
This proves Lemma \ref{reduction}.
\end{proof}

\begin{lemma}\label{kappa}
Suppose that the Kodaira dimension $\kappa := \kappa(X) > 0$. Then $r(G) \le \max\{0, \, n-1 -\kappa\}$.
This bound is optimal, by considering $(V \times W, \, G \times \id_W)$ with
$(V, G)$ as in \cite[Example 4.5]{DS04} and $W$ any projective manifold of general type.
\end{lemma}

\begin{proof}
If $\kappa = n$, then $|\Aut(X)| < \infty$; see \cite[Corollary 2.4]{NZ}.
So assume that $\kappa < n$.
Let $\Phi := \Phi_{|mK_X|} : X \ratmap \BPP(H^0(X, mK_X))$
be the Iitaka fibration and $Y$ the image of $\Phi$.
By \cite[Corollary 2.4]{NZ}, $G$ descends to a {\it finite} group $G|Y \le \Aut(Y)$.
Replacing $X$ and $Y$ by $G$-equivariant resolutions of indeterminacy locus
of $\Phi$ and singularities of $Y$, we may assume that
$\Phi$ is holomorphic and $Y$ is smooth; see also \cite[Lemma A.8]{NZ}.
Replacing $G$ by a finite-index subgroup and noting that $r(G) = r(G_1)$ when $|G:G_1| < \infty$,
we may assume that $G | Y = \{\id\}$, and for a general fibre
$F$ of $\Phi$ our $G | H^{1,1}(F)$ is Z-connected (and also solvable because:
some term $G^{(v)}$ of the derived series of $G$ acts trivially on $H^{1,1}(X)$
by the assumption, and the restriction $G^{(v)}|H^{1,1}(F)$ is $Z$-connected and acts trivially
on the restriction to $F$ of every K\"ahler class of $X$ so that
we can apply Lieberman \cite[Proposition 2.2]{Li} as in $(*)$ of Lemma \ref{kappa-} below).
We can identify $G$ with $G | F \subset \Aut(F)$.
By \cite[(2.1) Remark(11)]{Z2} or \cite[Theorem D]{NZ},
$N(G)|F = N(G|F)$.
By the first part of Theorem \ref{ThB}, our $G = G|F$ satisfies
$G/N(G) \cong \BZZ^{\oplus r}$ with $r \le \dim(F) - 1 = n - 1 - \kappa$.
\end{proof}

\begin{lemma}\label{kappa-}
Suppose that
$\kappa^{-1} := \kappa(X, -K_X) > 0$.
Then $r(G) \le n-2$.
\end{lemma}

\begin{proof}
Take $m >> 0$
so that the movable part $|S|$ of $|-mK_X|$
gives rise to a meromorphic map $\Psi : X \ratmap Z$
with $\dim Z = \kappa^{-1}$. Now $h^*S \sim S$ for all $h \in \Aut(X)$.
Replacing $X$ and $Z$ by their $G$-equivariant blowups,
we may assume that $|S|$ is base point free so that
$\Psi$ is holomorphic (but associated with a linear system
bigger than $|-mK_X|$),
$Z$ is smooth, and $G$ descends to $G|Z \le \Aut(Z)$.
If $\dim Z < \dim X$, then $r(G) \le n-2$ by Lemma \ref{reduction}.

Consider the case $\dim Z = \dim X$. Then
$S$ is nef and big,
whose class in $H^{1,1}(X, \BRR)$ is preserved by $G$,
so the argument of Lieberman \cite[Proposition 2.2]{Li} implies that $(*)$:
$|G : G \cap \Aut_0(X)| < \infty$; see \cite[Lemma 2.23]{Z2}
and also \cite[Lemma A.5]{NZ}.
Since every element of $\Aut_0(X)$ acts trivially on $H^*(X, \BZZ)$, it
is of null entropy and so does
every element of $G$. Thus $r(G) = 0$.
This proves Lemma \ref{kappa-}.
\end{proof}

\begin{lemma}\label{q}
Suppose that the irregularity $q(X) > 0$.
Then $r(G) \le n-2$, unless $X$ is bimeromorphic
to a complex torus.
\end{lemma}

\begin{proof}
Let $\alb_X : X \to \Alb(X)$ be the albanese map with
$q(X) = \dim \Alb(X)$. Let $Y := \alb_X(X)$ be the image.
If $\dim Y < n$, then $r(G) \le n-2$ by Lemma \ref{reduction}.

Suppose that $\dim Y = n$. Then $\alb_X$ is generically finite onto
its image and hence $\kappa(X) \ge \kappa(Y) \ge 0$,
since $Y$ is a subvariety of a complex torus
(cf. \cite[Lemma 10.1]{Ue}).
By Lemma \ref{kappa}, we may assume that $\kappa(X) \le 0$.
Thus $\kappa(X) = 0$. By \cite[Theorem 24]{Ka}, $\alb_X : X \to \Alb(X)$
is then surjective and bimeromorphic. This proves Lemma \ref{q}.
\end{proof}

\begin{lemma}\label{aut}
Suppose that $H$ is a positive-dimensional connected and closed subgroup
of $\Aut_0(X)$ so that $H$ is normalized by $G$, i.e.,
$G \le N_{\Aut_0(X)}(H)$. Then $r(G) \le n-2$, unless
some $H$-orbit is dense open in $X$.
\end{lemma}

\begin{proof}
By \cite[Lemma 4.2 (3), Theorem 4.1, Theorem 5.5]{Fu},
there is a quotient map $\pi: X \ratmap Y = X/H$
which is $G$-equivariant (so that $H$ acts on $Y$ trivially);
replacing $X, Y$ by $G$-equivariant blowups
as in \cite[(2.0)]{Fu},
we may assume that $\pi$ is holomorphic and
both $X$ and $Y$ are K\"ahler $G$-manifolds, because
$Y$ is in Fujiki's class $\SC$ and hence bimeromorphic to a
compact K\"ahler manifold (cf. Varouchas \cite{Va}).
Now just apply Lemma \ref{reduction}.
\end{proof}

\begin{lemma}\label{uniruled}
Suppose that $X$ is a uniruled projective manifold. Then
$r(G) \le n-2$, unless $X$ is rationally connected.
\end{lemma}

\begin{proof}
Applying the argument in \cite[Lemma 5.2]{NZ} and taking equivariant resolutions,
we may assume that the rationally connected fibration $X \to Y$ is holomorphic with $Y$ smooth
(and non-uniruled)
so that $G$ descends to $G|Y \subseteq \Aut(Y)$. If $\dim Y > 0$, then
$r(G) \le n-2$ by Lemma \ref{reduction}.
If $\dim Y = 0$, then $X$ is rationally connected.
\end{proof}

\begin{lemma}\label{algdim}
If the algebraic dimension $a(X) \in \{1, \dots, n-1\}$,
then $r(G) \le n-2$.
\end{lemma}

\begin{proof}
Let $X \ratmap Y$ be an algebraic reduction with $\dim Y = a(X)$ and connected fibres
(cf. \cite[Prop. 3.4]{Ue}).
Let $Y'$ be the main compact irreducible subvariety in the Douady space $D_X$ of $X$
parametrizing the general fibres of $X \ratmap Y$, and $X' \to Y'$ the restriction
of the universal family $Z \to D_X$.
As pointed out by Fujiki, $G$ acts biregularly on $X'$ and $Y'$
so that the holomorphic map $X' \to Y'$ is $G$-equivariant and bimeromorphic to $X \ratmap Y$.
Now just apply Lemma \ref{reduction} to the $G$-equivariant resolutions of $X'$ and $Y'$.
\end{proof}

We now complete the proof of Theorem \ref{ThB}.
By Lemmas \ref{kappa}, \ref{kappa-}, \ref{q} and \ref{algdim},
we may assume that $q(X) = 0$, $\kappa(X) = -\infty$, $\kappa^{-1}(X) \le 0$ and $a(X) \in \{0, n\}$.
Now $q(X) = 0$ implies that $\Aut_0(X)$ is a linear algebraic group;
see \cite[Corollary 5.8, Theorem 5.5]{Fu}.

Suppose that $H \le \Aut_0(X)$ is connected and closed and
has a dense open orbit in $X$. Then $X$ is almost homogeneous, Moishezon (and
hence projective, $X$ being K\"ahler)
and unirational, because $H$ is a rational variety by a result of Chevalley.
$X$ is birational to $\BPP^1 \times Z$ with $Z$ projecitve, by a result of H. Matsumura;
see \cite[Proposition 5.10 and its Remark]{Fu}.
The unirationality of $X$
implies that of $Z$; and a unirational surface is rational.

If $H$ is further commutative, then $X$ is rational because: the stabilizer $H_x$
of a general point of $x \in X$ is normal in $H$ and hence $H_x = H_y$
for all general $y \in X$ (by the density of the orbit $Hx$ in $X$),
so $H_x$ is trivial and
$H$ dominates $X$ birationally.

Now Theorem \ref{ThB} follows from Lemma \ref{uniruled}, the argument above and
the application of Lemma \ref{aut} to characteristic subgroups of $\Aut_0(X)$, especially to $\Aut_0(X)$
and to the last term (an abelian group) of the derived series of $R(\Aut_0(X))$ (the radical).

\vskip 1pc
Question \ref{Q1} below has been affirmatively answered in \cite[pp. 323-325]{DS04}
for abelian $G$ and without the condition $\Aut_0(X) = (1)$ (but this condition is
needed for non-abelian $G$
as seen in \cite[Example 4.5]{DS04}),
a short proof of which follows immediately from
Lemma \ref{aut} and Lieberman \cite[Proposition 2.2]{Li} just as in $(*)$ of Lemma \ref{kappa-}.

\begin{question}\label{Q1}
In Theorem $\ref{ThB}$, when the rank $r = n-1$ $($maximal$)$ and $\Aut_0(X) = (1)$,
can one say that $N(G)$ is a finite group?
\end{question}

\begin{question}
Can one find a rationally connected variety
$($or Calabi Yau manifold$)$
$X$ and a group $G \le \Aut(X)$
such that the rank in Theorem $\ref{ThB}$ satisfies $r =  \dim X -1$?
\end{question}


\begin{thebibliography}{99}

\bibitem
{Bi} G.~Birkhoff,
Linear transformations with invariant cones,
Amer.\ Math.\ Month. \textbf{74} (1967) 274--276.

\bibitem
{Di} T. -C.~Dinh,
Suites d'applications m\'eromorphes multivalu\'ees et
courants laminaires,
J.\ Geom.\ Anal. \textbf{15}  (2005) 207--227.

\bibitem
{DS04} T.-C.~Dinh and N.~Sibony,
Groupes commutatifs d'automorphismes d'une vari\'et\'e k\"ahlerienne compacte,
Duke Math.\ J. \textbf{123} (2004) 311--328.

\bibitem
{DS05} T.-C.~Dinh and N.~Sibony,
Green currents for holomorphic automorphisms of compact K\"ahler
manifolds,
J.\ Amer.\ Math.\ Soc. \textbf{18} (2005) 291--312.

\bibitem
{Fu} A.~Fujiki,
On automorphism groups of compact K\"ahler manifolds,
Invent.\ Math. \textbf{44} (1978) 225--258.

\bibitem
{Fr} S.~Friedland,
Entropy of polynomial and rational maps, Ann.\ of Math.\
\textbf{133} (1991) 359--368.

\bibitem
{Gr} M.~Gromov,
On the entropy of holomorphic maps,
Enseign.\ Math. \textbf{49} (2003) 217--235.

\bibitem
{Ka} Y.~Kawamata,
Characterization of abelian varieties,
Compos.\ Math. \textbf{43} (1981) 253--276.

\bibitem
{KOZ} J.~Keum, K.~Oguiso and D. -Q.~Zhang,
Conjecture of Tits type for complex varieties and
theorem of Lie-Kolchin type for a cone, Math.\ Res.\ Lett. to appear,
also: arXiv:math/\textbf{0703103}.

\bibitem
{Ki} J.~H.~Kim,
Solvable automorphism groups of a compact Kaehler manifold,
arXiv:\textbf{0712.0438v6}.

\bibitem
{KM} J.~Koll\'ar and S.~Mori,
Birational geometry of algebraic varieties,
Cambridge Tracts in Math. \textbf{134},
Cambridge Univ.\ Press, 1998.

\bibitem{Li}
D.~I.~Lieberman,
Compactness of the Chow scheme: applications to automorphisms
and deformations of K\"ahler manifolds,
\emph{Fonctions de plusieurs variables complexes, III}
(\emph{S\'em.\ Fran\c{c}ois Norguet, 1975--1977}), pp.~140--186,
Lecture Notes in Math., \textbf{670}, Springer, 1978.

\bibitem
{NZ} N.~Nakayama and D. -Q.~Zhang,
Building blocks of \'etale
endomorphisms of complex projective manifolds, RIMS preprint
\textbf{1577}, Res.\ Inst.\ Math.\ Sci.\ Kyoto Univ.\ 2007.

\bibitem
{Og} K.~Oguiso,
Tits alternative in hypek\"ahler manifolds,
Math.\ Res.\ Lett. \textbf{13} (2006) 307--316.

\bibitem
{Ti} J.~Tits,
Free subgroups in linear groups,
J. Algebra {\bf 20} (1972) 250--270.

\bibitem
{Ue} K.~Ueno,
\emph{Classification theory of algebraic varieties and compact complex spaces},
Lecture Notes in Mathematics, \textbf{439},
Springer, 1975.

\bibitem{Va}
J.~Varouchas,
Stabilit\'e de la classe des vari\'et\'es k\"ahleriennes
pour les certains morphisms propres,
Invent.\ Math. \textbf{77} (1984) 117--128.

\bibitem
{Yo} Y.~Yomdin,
Volume growth and entropy,
Israel J.\ Math. \textbf{57} (1987) 285--300.

\bibitem
{Z1} D. -Q.~Zhang,
Automorphism groups and aniti-pluricnonical curves,
Math.\ Res.\ Lett. \textbf{15} (2008) 163--183.

\bibitem
{Z2} D. -Q.~Zhang,
Dynamics of automorphisms on projective complex manifolds,
J.\ Differential Geom.\ to appear, also: arXiv:{\bf 0810.4675}.

\bibitem
{Z3} D. -Q.~Zhang,
Dynamics of automorphisms of compact complex manifolds,
{\it Proceedings of The Fourth International Congress of Chinese Mathematicians (ICCM2007)},
Vol.~\textbf{II}, pp. 678 - 689, also: arXiv:\textbf{0801.0843}.

\end{thebibliography}
\end{document}